\documentclass[leqno,11pt]{amsart}
\usepackage{amssymb,amsthm,amsmath,txfonts,a4wide}
\usepackage{xcolor}
\usepackage{enumerate}
\usepackage{hyperref}

\newtheorem{defn}{Definition}[section]

\newtheorem{lem}[defn]{Lemma}
\newtheorem{thm}[defn]{Theorem}
\newtheorem{prop}[defn]{Proposition}
\newtheorem{cor}[defn]{Corollary}
\newtheorem{rem}[defn]{Remark}

\theoremstyle{definition}

\numberwithin{equation}{section}

\let\ptd=\partial

\def\bE{\mathbb{E}}

\def\bP{\mathbb{P}}
\def\bR{\mathbb{R}}
\def\bT{\mathbb{T}}

\def\cA{\mathcal{A}}

\def\cH{\mathcal{H}}

\let\D=\Delta
\let\ep=\epsilon

\let\Si=\Sigma

\let\pa=\partial
\let\ptd=\partial

\newcommand{\fr}[2]{\frac{#1}{#2}}

\begin{document}
    \title{Generic ill-posedness for wave equation of power type on 3D torus }
    \author{Bo XIA}
    \address{Laboratoire de Math\'ematiques, University Paris-Sud 11, F-91405.}
    \email{bo.xia@math.u-psud.fr}
    \thanks{The author is supported by CSC.}
    \maketitle

    \begin{abstract}
      In this article, we prove that the equation
      \begin{equation*}
        \left\{\begin{split}
            &(\ptd^2_t-\Delta)u+|u|^{p-1}u=0,\ \ \ 3\leq p<5\\
            &\big(u(0),\ptd_tu(0)\big)=(u_0,u_1)\in H^{s}(\bT^3)\times H^{s-1}(\bT^3)=:\cH^s(\bT^3)
        \end{split}\right.
    \end{equation*}
    with $s<\frac{3}{2}-\frac{2}{p-1}$ is everywhere ill-posed. This work also indicates that, only properly regularizing the initial data can we smoothly approximate the solutions constructed in \cite{BT14} and \cite{Xia14}.
    \end{abstract}
\section{Introduction}
    In this article we consider the semi-linear wave equation
    \begin{equation}\label{eqn:p}
        \left\{\begin{split}
            &(\ptd^2_t-\Delta)u+|u|^{p-1}u=0,\ \ \ 3\leq p<5\\
            &\big(u(0),\ptd_tu(0)\big)=(u_0,u_1)\in H^{s}(\bT^3)\times H^{s-1}(\bT^3)=:\cH^s(\bT^3)
        \end{split}\right.
    \end{equation}
and we are interested in the ill-posedness issues in the classical Hardamard sense for the super-critical case $0<s<\fr{3}{2}-\fr{2}{p-1}$. We obtained that the set of data, initiated at which Equation (\ref{eqn:p}) is ill-posed, is dense in $\cH^s(\bT^3)$. To explain this clearly, we first review some development for the Cauchy problem for the equation \eqref{eqn:p}.

There is an extensive literature, in the past thirty years, dedicated to the well-posedness of the equation \eqref{eqn:p} in the (sub-)critical case $s\geq\fr{3}{2}-\fr{2}{p-1}$. The first rigorous proof of the local existence and uniqueness for the (sub-)critical equation \eqref{eqn:p} was obtained by Lindblad-Sogge in \cite{LS95} by using Strichartz estimates. After this, there arose a lot of works on the scattering theory and the growth of Sobolev norms of solutions to  Equation (\ref{eqn:p}). However, the Cauchy problem is still quite open in the super-critical case. To this end, Lebeau \cite{Leb} first proved, based on the local-in-time asymptotic analysis of geometric optics, that the solutions to Equation (\ref{eqn:p}) with $p$ any odd integer not smaller than $7$ are unstable. And then, Christ-Colliander-Tao via a small dispersion analysis in \cite{CCT03} and  Burq-Tzvetkov via a direct ODE approximation proved in \cite{BT07a} that the equation \eqref{eqn:p} is not well-posed by contradicting the continuous dependence on the initial data around the origin $(u_0,u_1)=(0,0)\in\cH^s(\bT^3)$. For more information on other approaches to obtain ill-posedness results for other dispersive equations, see \cite{kenig2001ill} and \cite{christ2003asymptotics}.

Now that the super-critical equation is ill-posed, in what sense can we have the well-posedness result similar to that of the (sub-)critical equation? In order to answer this question, by introducing the concept "probabilistic Hardamard well-posedness", Burq-Tzvetkov proved the local well-posedness result in \cite{BT07a} and the global well-posedness result in \cite{BT14} for the cubic case, which is generalized in \cite{Xia14} to the general power $3\leq p<5$ with the restricted regularity $\frac{2(p-3)}{p-1}<s<\fr{3}{2}-\fr{2}{p-1}$. Precisely, these results can be summarized as
\begin{thm}[\cite{BT14},\cite{Xia14}]\label{thm:pro}
    Let $1>s>\frac{2(p-3)}{p-1}$ for $3\leq p<5$ and any $(u_0,u_1)\in\cH^s(\bT^3)$. Then there exist a probability measure $\mu$ depending on $(u_0,u_1)$, which is supported on $\cH^s(\bT^3)$ and is of positive probability on any non-empty open subset of $\cH^s(\bT^3)$, and a subset $\Sigma\subset\cH^s(\bT^3)$ of full $\mu$-measure such that for any $(v_0,v_1)\in\Sigma$, there is a unique global solution $v$ to the nonlinear equation
    \begin{eqnarray}\label{eqn.proba.main}
        (\ptd^2_t-\Delta)v+|v|^{p-1}v=0,\ \ (v(0),\ptd_tv(0))=(v_0,v_1)
    \end{eqnarray}
    satisfying
    \[
        \big(v(t),\ptd_tv(t)\big)\in (S(t)(v_0,v_1),\ptd_tS(t)(v_0,v_1))+C(\bR_t;\cH^1)\subset C(\bR_t;\cH^s).
    \]
    Furthermore, we can specially choose $\mu$ such that the flow $\Phi(t)$ generated by the equation \eqref{eqn.proba.main} is continuous in the following sense: for any $\ep>0,\ \tilde{T}>0$, there exists $\eta>0$ such that
    \begin{multline*}
      \mu\otimes\mu\Big((V_0,V_1)\in \cH^s\times\cH^s\colon \|\Phi(t)(V_0)-\Phi(t)(V_1)\|_{X_{\tilde{T}}}>\ep\big| \\ \|V_0-V_1\|_{\cH^s}\leq \eta \text{ and } (V_0,V_1)\in B_{\Lambda}\times B_{\Lambda}\Big)\leq g(\ep,\eta),
    \end{multline*}
    where $X_{\tilde{T}}\equiv\big(C([0,\tilde{T}];H^s)\cap L^{\frac{2p}{p-3}}([0,T];L^{2p})\big)\times C([0,\tilde{T}];H^{s-1})$, $B_{\Lambda}\equiv\big\{V\in\cH^s\colon\|V\|_{\cH^s}\leq \Lambda\big\}$, and $g(\ep,\eta)$ satisfies
    \[
        \lim_{\eta\rightarrow 0}g(\ep,\eta)=0,\ \ \ \forall \ep>0.
    \]
  \end{thm}
  \begin{rem}
    Recently, Chenmin Sun and the author in \cite{C-X15}, by following the ideas used by Oh-Pocovnicu \cite{OhPo} solving energy-critical wave equation, improved the regularity $s>\frac{2(p-3)}{p-1}$ required to solve Equation (\ref{eqn:p}) to $s>\frac{p-3}{p-1}$, but we do not have the probabilistic dependence on the initial data there.  
  \end{rem}
  Theorem \ref{thm:pro} says that we can construct a probability on proper function space (e.g. $H^s\times H^{s-1}$ in our case) and show that the wave equation (\ref{eqn:p}) is globally well-posed almost surely and moreover the flow generated by the equation is conditionally continuous.  Interestingly, as pointed out by Burq-Tzvetkov, it is also possible to show that these solutions are limits of smooth solutions obtained by properly regularized initial data. In fact we have
\begin{prop}\label{rem:w.p.}
  Given $(v_0,v_1)\in\Si$, where $\Si$ is the invariant set in \ref{thm:pro}. We truncate $(v_{0,n},v_{1,n})=(\chi(2^{-|n|}|\D|)v_0,\chi(2^{-|n|}|\D|)v_1)$ with $\chi$ a radial bump function equal to $1$ in the unit interval $[0,1]$ and vanishing outside $[0,2]$, then almost surely, the solutions $v_n(t)$ to the equation \eqref{eqn.proba.main} issued from $(v_{0,n},v_{1,n})$ converge, as $n$ tends to infinity, to the solution $v(t)$ with initial data $(v_0,v_1)$ in the space $C(\bR,\cH^s)$. Consequently we have that $v_n(t)$ converges almost surely to $v(t)$ in $C(\bR,\cH^s)$ as $n\rightarrow\infty$.
\end{prop}

Indeed, to prove this statement, we only need to apply the Burq-Tzvetkov strategy (see \cite{BT14} or \cite{Xia14}) to the equation satisfied by $w_n=v-v_n$
\begin{equation*}
  \left\{
        \begin{split}
            &(\ptd^2_t-\Delta)w_n+|v|^{p-1}v-|v_n|^{p-1}v_n=0 \\
            &\big(w_n(0),\ptd_tw_n(0)\big)=\big((1-\chi(2^{-|n|}\D))v_0,(1-\chi(2^{-|n|}\D))v_1\big),
          \end{split}
          \right.
    \end{equation*}
    and we can establish that $\mu\big((v_0,v_1)\in\Si\colon \|w_n\|_{C(\bR,\cH^s)}\rightarrow_{n\rightarrow\infty}0\big)=1$.

    Now a natural question coming up is the stability of this regularizing process. At a first sight, it might be difficult to define this stability, so we ask this question in a slightly different manner: what's happening when we are using a different regularizing procedure? The purpose of this work is to address this question, and indeed, via the ODE approach used in \cite{BT07a}, we obtain
    \begin{thm}\label{cor:1}
    Let $s\in(0,\fr{3}{2}-\fr{2}{p-1})$ fixed and $(u_0,u_1)\in\cH^s$. Then for any $\epsilon>0$, there exists a sequence $\big((u_{0n},u_{1n})\big)^{\infty}_{n=1} \in C^{\infty}_c$ converging to $(u_0,u_1)$ in $\cH^s$ such that the solutions $\big(u_n(t)\big)$ to the equations
    \begin{equation*}
        \begin{cases}
            &(\ptd^2_t-\Delta)u_n+|u_n|^{p-1}u_n=0\\
            &\big(u_n(0),\ptd_tu_n(0)\big)=(u_{0n},u_{1n})
        \end{cases}
    \end{equation*}
    satisfy
    \[
        \|u_n\|_{L^{\infty}\big([0,\epsilon],\cH^s\big)}\rightarrow_{n\rightarrow+\infty}+\infty.
    \]
\end{thm}
The proof of Theorem \ref{cor:1} is just a combination of diagonal argument with the following proposition \ref{thm:main}, so we omit it here.
\begin{prop}\label{thm:main}
    Let $s\in(0,\fr{3}{2}-\fr{2}{p-1})$ fixed and $(u_0,u_1)\in C^{\infty}_c\times C^{\infty}_c(M)$ arbitrarily  given and $u(t)$ be its corresponding smooth solution, then for any $\epsilon>0$, there exist $\delta>0$ and a sequence $\big(u_n(t)\big)^{\infty}_{n=1}$ of $C^{\infty}_{c}(M)$ functions such that
    \begin{equation*}
        \begin{cases}
            &(\ptd^2_t-\Delta)u_n+|u_n|^{p-1}u_n=0\\
            &\big(u_n(0),\ptd_tu_n(0)\big)=(u_{0n},u_{1n})
        \end{cases}
    \end{equation*}
    with
    \[
        \|(u_{0n},u_{1n})-(u_0,u_1)\|_{\cH^s}\rightarrow_{n\rightarrow+\infty}0,
    \]
    but
    \[
        \|u_n-u\|_{L^{\infty}\big([0,\epsilon],\cH^s\big)}\rightarrow_{n\rightarrow+\infty}+\infty.
    \]
    In particular, we have that
    \[
        \|u_n\|_{L^{\infty}([0,\epsilon],\cH^s)}\rightarrow_{n\rightarrow\infty}\infty.
    \]
\end{prop}


    Proposition \ref{thm:main} and Theorem \ref{cor:1} say that the equation \eqref{eqn:p} is ill-posed in the space $L^{\infty}([0,\epsilon],\cH^s)$ for any fixed $\epsilon>0$. Actually, for the equation \eqref{eqn:p}, there exists initial data such that the $H^s$-energy blows up instantaneously, as is shown by the following construction:
\begin{thm}
    Let us fix $s\in(0,\fr{3}{2}-\fr{2}{p-1})$, then for any compactly supported data $(u_0,u_1)\in C^{\infty}_c\times C^{\infty}_c$, there exists an initial data $(f_0,f_1)\in \cH^s\backslash C^{\infty}_c$ arbitrarily close to $(u_0,u_1)$ in $\cH^s$, such that the equation \eqref{eqn:p}, complemented by the initial condition $\big(u(0),\partial_tu(0)\big)=(f_0,f_1)$ and satisfying in addition the finite speed of propagation, has no solution in $L^{\infty}([0,T],\cH^s),\ T>0$.
\end{thm}
 
\begin{rem}
  This theorem states that the set of datum initiated at which the equation \eqref{eqn:p} is ill-posed in $L^{\infty}([0,T],\cH^s),\ T>0$ is dense in $\cH^s$. Indeed what we expect is that it (at least) has a $G_{\delta}$-structure, but due to technical difficulties, we cannot prove this now.
\end{rem}

\begin{rem}\label{rem:1}
    By combining the Theorem \ref{thm:pro}, we obtain the following 'almost-sure non-continuous dependence on the initial data':
    Let $\frac{2(p-3)}{p-1}<s<\fr{3}{2}-\fr{2}{p-1}$ for $3\leq p<5$ and any $(u_0,u_1)\in\cH^s(\bT^3)$ be given. For any given $\varepsilon>0$, there exist a probability measure $\mu$ depending on $(u_0,u_1)$, which is supported on $\cH^s(\bT^3)$ and is of positive probability on any non-empty open subset of $\cH^s(\bT^3)$, and a subset $\Sigma\subset\cH^s(\bT^3)$ of full $\mu$-measure such that for any $(v_0,v_1)\in\Si$ with its corresponding solution $v(t)$, there exists a sequence $\big(v_n(t)\big)^{\infty}_{n=1}$ of $C^{\infty}(\bT^3)$ functions such that
    \begin{equation*}
        \begin{cases}
            &(\ptd^2_t-\Delta)v_n+|v_n|^{p-1}v_n=0\\
            &\big(v_n(0),\ptd_tv_n(0)\big)=\big(v_{0n},v_{1n}\big)
        \end{cases}
    \end{equation*}
    with
    \[
        \|(v_{0n},v_{1n})-(v_0,v_1)\|_{\cH^s(\bT^3)}\rightarrow_{n\rightarrow+\infty}0,
    \]
    but
    \[
        \|v_n(t)-\Phi(t)(v_{0n},v_{1n})\|_{L^{\infty}([0,\varepsilon]H^s(\bT^3))}\rightarrow_{n\rightarrow+\infty}+\infty.
    \]

    We should note that this result does not contradict the probabilistic continuity on the initial datum stated in Theorem \ref{thm:pro}, which says that the set of the datum, initiated at which the equation \eqref{eqn.proba.main} does continuously depend on the data, is of $\mu$-probability $1$, and which does not exclude the possibility that there may exist discontinuity on the data. This remark convinces us that there actually not only exists discontinuity on the initial data, but also {the event}, consisting of the data initiated at which such discontinuity occurs, is of $\mu$-probability $1$.
\end{rem}
	We also have a deterministic analogue to results as in Remark \ref{rem:1}.
\begin{rem}
	Let $\frac{2(p-3)}{p-1}<s<\fr{3}{2}-\fr{2}{p-1}$ for $3\leq p<5$ and any $(v_0,v_1)\in\cH^s(\bT^3)$ be given. For any given $\varepsilon>0$, there exists a sequence $\big(v_n(t)\big)^{\infty}_{n=1}$ of $C^{\infty}(\bT^3)$ functions such that
    \begin{equation*}
        \begin{cases}
            &(\ptd^2_t-\Delta)v_n+|v_n|^{p-1}v_n=0\\
            &\big(v_n(0),\ptd_tv_n(0)\big)=\big(v_{0n},v_{1n}\big)
        \end{cases}
    \end{equation*}
    with
    \[
        \|(v_{0n},v_{1n})-(v_0,v_1)\|_{\cH^s(\bT^3)}\rightarrow_{n\rightarrow+\infty}0,
    \]
    but
    \[
        \|v_n(t)\|_{H^s(\bT^3)}\rightarrow_{n\rightarrow+\infty}+\infty.
    \]
\end{rem}

The article proceeds as follows: in Section \ref{sec:series} we present the similar results as above for the non-evolutionary case, and then in Section \ref{sec:lwe} generalize these results to the linear wave equation. For the nonlinear case, in Section \ref{sec:cnwe}, we summarize the routine in which Burq-Tzvetkov proved the ill-posedness result for cubic wave equation, and then by a similar argument, we prove the ill-posedness result for fixed time interval for equation \eqref{eqn:p} with generic data in Section \ref{sec:pnlw}, and then in section \ref{sec:inst}, we prove the set of datum initiated at which the equation becomes ill-posed instantaneously is dense in $\cH^s$.

\section*{Acknowledgement} I should thank Nicolas Burq for his encouraging and carefully advising when I was working on this problem.

\section{The series case}\label{sec:series}
We begin our discussion with recalling some results on the randomization of the initial data used in \cite{BT07a} and references therein. These typical results, especially Proposition \ref{prop:probnottotopo}, serves as the prototype in our consideration in the PDE case.

     Suppose the function $u$ on the torus $\bT^3$ is given by its Fourier series
\[
    u(x)=\sum_nu_ne^{in\cdot x}.
\]
We can randomize it via
    \begin{eqnarray}\label{randomization}
    u^{\omega}(x)=\sum_nu_n\alpha_n(\omega)e^{in\cdot x},
    \end{eqnarray}
where $\big(\alpha_n(\omega)\big)_n$ is a series of i.i.d real-valued standard Gaussian random variable on the probability space $(\Omega,\cA,\bP)$. Now we have the following statement.

\begin{thm}\cite{BT07a}\cite{paley1930some}\label{l2largeprobability}
    If $u\in L^2(\bT^3)$, then almost surely $u^{\omega}\in L^q(\bT^3)$ for any $q>2$.
  \end{thm}
  For readers' convenience, and also for the self-containing of this article, we present the proof of this theorem. One also can refer to \cite{BT07a}\cite{paley1930some}.
  
\begin{proof}
    We first consider the case $q=2k$ for some positive integer. Now we can calculate the expectation
    \begin{eqnarray*}
        \bE\Big(\|u^{\omega}\|^q_{L^q_x}\Big) &=&\int_{\Omega}\|u^{\omega}\|^q_{L^q}d\bP\\
                                              &=&\int_{\Omega}\int_{\bT^3}\sum_{2\beta_1+\cdots+2\beta_i=2k}\frac{q!}{(2\beta_1)!\cdots(2\beta_i)!}\Pi(u_{n_i}\alpha_{n_i})^{2\beta_i}dxd\bP\\
                                              &\leq&\int_{\bT^3}\sum_{\beta_1\cdots+\beta_i=k}\frac{(2k)!}{(2\beta_1)!\cdots(2\beta_i)!}\bE\Big(\Pi(u_{n_i}\alpha_{n_i})^{2\beta_i}\Big)dx\\
                                              &\leq& C\Big(\sum|u_n|^2\Big)^k = C \|u\|^{2k}_{L^2}.
    \end{eqnarray*}
    By H\"older inequality, we have
    \[
        \bE\Big(\|u^{\omega}\|_{L^q}\Big)\leq \bE\Big(\|u^{\omega}\|^q_{L^q}\Big)^{1/q}\leq c\|u\|_{L^2}.
    \]
    Next for the case $2k<q<2k+2$, by the interpolation of $L^p$-spaces and H\"older inequality with $\theta_1+\theta_2=1$ and $\fr{\theta_1}{2k}+\fr{\theta_2}{2k+2}=\fr{1}{q}$, we have
    \begin{eqnarray*}
        \bE\Big(\|u^{\omega}\|_{L^q}\Big)&\leq&\bE\Big(\|u^{\omega}\|^{\theta_1}_{L^{2k}}\|u^{\omega}\|^{\theta_2}_{L^{2k+2}}\Big)\\
                                         &\leq&\Big[\bE\Big(\|u^{\omega}\|_{L^{2k}}\Big)\Big]^{\theta_1}\Big[\bE\Big(\|u^{\omega}\|_{L^{2k+2}}\Big)\Big]^{\theta_2}\\
                                         &\leq& c\|u\|_{L^2},
    \end{eqnarray*}
    which completes the proof by using Markov-Chebychev's inequality.
\end{proof}

As a corollary of the the proof of Theorem \ref{l2largeprobability}, we have
\begin{cor}
    If $(u_n)\in l^2$, then for any $q>2$, almost surely $u^{\omega}\equiv \sum u_n\alpha_n(\omega)e^{in\cdot x}\in L^q_x.$
\end{cor}

It is well known in the subject of real analysis that the Cantor set is an open set with no interior points, but it still has measure $1$. A little bit contrary to this, even though the set $(u^{\omega})$ is contained in $L^q$ with a large probability, there is still a large part of elements which do not belong to $L^q$, as is shown

\begin{prop}\footnote{  Actually, this is an exercise, left by N. Burq when he gave the M2 class 'Super-critical nonlinear wave equations' in the winter of the year 2012.}\label{prop:probnottotopo}
    For any $p>2$, there exists a dense $G_{\delta}$ type subset $G\subset L^2$ such that for any $v=\sum_nv_ne^{in\cdot x}\in G$, the series $v=\sum_nv_ne^{in\cdot x}$ is NOT lying in the space $L^p$.
  \end{prop}

\begin{proof}
\begin{itemize}
    \item{Step 1.}
    Now given $u=\sum_nu_ne^{in\cdot x}\in L^2$, for any $\epsilon>0$, we can choose $N$ large enough such that
    \[
        \|u-v_0\|_{L^2}<\frac{\epsilon}{2}.
    \]
    Denote $\chi\in C^{\infty}_c(\bR^3)$ such that $\chi$ is radial and equal to 1 when $|x|\leq 1/2$ and is 0 when $|x|>1$. We also denote $\chi(x)=\chi(|x|)$. Set $w_k(x)=k^{d/2}\chi\big(k(x-x_0)\big)$. Then we have
    \[
        \|w_k\|_{L^p}\sim k^{d(\frac{1}{2}-\frac{1}{q})},
    \]
    and $\textrm{supp}w_k\subset[x_0-\frac{1}{k},x_0+\frac{1}{k}]$, where we admit that the addition of a number to one vector is just to add to the first component. Therefore the support of a function of the form $w_{k_p}(x-x_p)$ with $x_p=\sum_{i=1}^p2^{-(i+2)}$ and $k_p=2^{p+3}\times 2^{10}$ does not intersect that of any other one of different $p$. Now for $\ep_p=2^{-(\frac{1}{2}-\frac{1}{q})}\sqrt{\ep}$, we set
    \[
        v\equiv v_0+\sum_{p=1}^{\infty}\ep_pw_{k_p}(x-x_p).
    \]
    Then thanks to the disjointness of the supports of functions $w_{k_p}$ for different $k_p$, we have
    \[
        \|u-v\|^2_{L^2}\leq \|u-v_0\|^2_{L^2}+\sum\ep^2_p=C\ep,
    \]
    but
    \[
        \|\sum_{p=1}^{\infty}\ep_pw_{k_p}(x-x_p)\|^q_{L^q}\sim \sum\ep^q_pk_p^{d(\frac{1}{2}-\frac{1}{q})}=\sum\big(2^{4d(\frac{1}{2}-\frac{1}{q})}\ep\big)^q=+\infty.
    \]

    \item{Step 2.} Now we define the set
        \[
            G_M\equiv\{v \in L^2\colon \|\Pi_Mv\|_{L^q}>\log\log M \text{ for } M \text{ large} \}
        \]
    Now define $G\equiv \lim\sup G_M$. By the arbitrary choice of $u$ in Step 1, we can see that $G$ is dense in $L^2$. Now we only need to see the openness of $G_M$ in $L^2$, which is obvious by the continuity of the $L^q$ norm together with the continuity property of the projection onto the low frequencies.
\end{itemize}
\end{proof}

\section{The case of linear wave equation}\label{sec:lwe}
In this section, by using the ideas used in the last section, we are going to present some results similar to Proposition \ref{prop:probnottotopo} for the wave equation
\begin{eqnarray}\label{lwe}
        \left\{
            \begin{split}
            & \big(\pa_t^2-\Delta\big)u =0,\\
            & (u(0),\pa_tu(0))=(u_0, 0), u_0\in L^2.
            \end{split}
        \right.
\end{eqnarray}
It is known that the equation is well-posed in $L^2$, but it is ill-posed in the $L^p$ space, as can be shown by contradicting the continuous dependence on the initial data.

\begin{thm}\label{prop:probnot2topo:lwe}
    For any $p>2$, there exists a $G_\delta$ type set $G$ dense in $L^2$, such that the equation \eqref{lwe} is not well-posed in $C([0,T];L^p)$ for any $T>0$, no matter how small it is.
\end{thm}
\begin{rem}
    For the proof, we only list the essential part. And we omit the construction of such a set, which is similar to the series case.
\end{rem}
\begin{proof}
    Let $\ep_p,w_{k_p}$ be the same as in the proof of Proposition \ref{prop:probnottotopo}, and set $v_0\equiv \Pi_{\leq N}u_0+\sum_{p=1}^Mw_{k_p}(x-x_p)$ and hence $\|u_0-v_0\|_{L^2}\leq c\ep$. The essential idea underlying this proof is that the time is so short that it does not destroy the property that the supports of $w_{k_p}$ do not intersect with each other for different $k_p$. Let us consider the Cauchy problem
\begin{eqnarray}\label{lwe}
        \left\{
            \begin{split}
            & \big(\pa_t^2-\Delta\big)v =0,\\
            & (v(0),\pa_tv(0))=(v_0, 0).
            \end{split}
        \right.
\end{eqnarray}
    Now set $t_M=2^{-(M+10)}$, then by the finite speed of propagation of waves, we have
    \[
        \|u(t_M,\cdot)-v(t_M,\cdot)\|_{L^2}\leq \|cos(t\sqrt{-\Delta})\Pi_{>N}u_0\|_{L^2}+\|\sum_{p=1}^Mcos(t\sqrt{-\Delta})\ep_pw_{k_p}(\cdot-x_p)\|_{L^2}\leq
        c\ep,
    \]
    and for any $q>2$
    \[
        \|\sum_{p=1}^Mcos(t\sqrt{-\Delta})\ep_pw_{k_p}(\cdot-x_p)\|^q_{L^q}\sim\sum_1^M(2^{4d(\frac{1}{2}-\frac{1}{q})})^q\ep^{q}.
    \]
    Therefore, as M tends to $\infty$, we have
    \[
        \|u(t_M,\cdot)-v(t_M,\cdot)\|_{L^q}\geq \sum_1^M(2^{4d(\frac{1}{2}-\frac{1}{q})})^q\ep^{q}\rightarrow \infty.
    \]
    which finishes the proof.
\end{proof}

\section{The case of cubic wave equation}\label{sec:cnwe}
    In the following, we are going to show the equation \eqref{eqn:p} is ill-posed for the cubic wave equation with generic initial datum. But first we should recall the ill-posedness result for the cubic wave equation with zero initial data by N. Burq and N. Tzvetkov \cite{BT07a}. Precisely
\begin{prop}\label{illposed:zerodata}
    Let us fix $s\in ]0,1/2[$. Then there exists $\delta>0$ and a sequence $(t_n)$ of positive numbers tending to zero and a sequence $(u_n(t))$ of $C^{\infty}_c(M)$ functions such that
    \[
        (\partial^2_t-\Delta)u_n+u^3_n=0
    \]
    with
    \[
        \|u_n(0)\|_{\cH^s(M)}\leq C\log(n)^{-\delta}\rightarrow_{n\rightarrow\infty}0
    \]
    but
    \[
        \|u_n(t_n)\|_{\cH^s(M)}\geq C\log(n)^{\delta}\rightarrow_{n\rightarrow\infty}\infty.
    \]
\end{prop}
We here outline the proof of this proposition. The basic idea is to compare the solutions to the equation
    \begin{eqnarray}\label{eqn:cubic:series}
        \left\{
            \begin{split}
                &(\partial^2_t-\Delta)u_n+u^3_n=0\\
                &(u_n(0),\partial_tu_n(0))=(f_{1,n}(x),0)
            \end{split}
        \right.
    \end{eqnarray}
and these to the ODEs
    \begin{eqnarray}\label{eqn:ode:series}
        \left\{
            \begin{split}
               & v''_n+v^3_n=0 \\
               & (v_n(0),v'_n(0))=(f_{1,n}(x),0).
            \end{split}
        \right.
    \end{eqnarray}
Under a special choice of the initial data $f_{1,n}=\kappa_nn^{3/2-s}\phi(nx)$ with $\phi$ a nontrivial bump function on $\bR^3$, the solutions to the ODEs \eqref{eqn:ode:series} have an explicit representation $v_n(t,x)=\kappa_nn^{3/2-s}V(tk_nn^{3/2-s}\phi(nx))$, where $V$ solves
the ODE
    \begin{eqnarray}\label{eqn:ode}
        V''+V^3=0,\ \  V(0)=1,\ \ V'(0)=0.
    \end{eqnarray}
(Although the solution to the ODE \eqref{eqn:ode} can be represented explicitly with the help of Jacobian elliptic functions, we do not need to, since what we need is just the periodicity property of the solution.)
Then the following two basic facts finish the proof.
    \begin{itemize}
        \item $u_n$ and $v_n$ are very close to each other with respect to the semi-classical energy $E_n(u)$ defined by
            \[
                E_n(u)\equiv n^{-(1-s)}\sqrt{\|\partial_tu\|^2_{L^2}+\|\nabla u\|^2_{L^2}}+n^{-(2-s)}\sqrt{\|\partial_tu\|^2_{H^1}+\|\nabla u\|^2_{H^1}}
            \]
            and consequently $\|u_n-v_n\|_{H^s}\leq Cn^{-\epsilon}$; ( It is this fact that requires the regularity $s$ should be smaller that $1/2$.)
        \item Due to the periodicity of the solution to the ODE \eqref{eqn:ode}, by the explicit representation, we can do some calculations, which lead to $\|v_n\|_{H^s}\rightarrow\infty$ as $n\rightarrow \infty$. (It is this fact that requires $s>0$.)
    \end{itemize}

    Generally, we should consider the cubic wave equation of general datum  $(u_0,u_1)\in C^{\infty}_c\times C^{\infty}_c$
    \begin{eqnarray}\label{eqn:cubic}
        \left\{
            \begin{split}
                &(\partial^2_t-\Delta)u+u^3=0\\
                &(u(0),\partial_tu)=(u_0,u_1) \in C^{\infty}_c\times C^{\infty}_c.
            \end{split}
        \right.
    \end{eqnarray}
    Our goal is to show that the solution to generally ill-posed in $\cH^s$ for $s\in]0,1/2[$.

    In order to easily get the $H^s$ norm blow-up, we just try to add the solutions of ODEs
    \begin{eqnarray}\label{eqn:ode:series:strategy3}
    \left\{
        \begin{array}{rcl}
            \frac{d^2}{dt^2}v_n+v^3_n &=& 0\\
            (v_n(0),\frac{d}{dt}v_n(0))&=&(\psi_n(x),0)
        \end{array}
    \right.
    \end{eqnarray}
to these of the PDEs
    \begin{eqnarray}\label{eqn:cubic:series:strategy3}
    \left\{
        \begin{array}{rcl}
            (\partial^2_t-\Delta)u_n+u^3_n &=&0\\
            (u_n(0),\partial_tu_n(0)) &=& (u_0+\psi_n,u_1).
        \end{array}
    \right.
    \end{eqnarray}
And this leads to the difference equation for $w_n=u_n-v_n$
    \begin{eqnarray}\label{eqn:cubic:diff:strategy3}
    \left\{
        \begin{array}{rcl}
            (\partial^2_t-\Delta)w_n+w_n^3 &=& \Delta v_n-3w_n^2v_n-3w_nv^2_n\\
            (w_n(0),\partial_tw_n(0)) &=& (u_0,u_1)
        \end{array}
    \right.
    \end{eqnarray}
In this case, thanks to the smallness of the semi-classical energy $E_n(w_n(0))\sim{n}^{-(1-s)}$, one can follow the strategy listed above to obtain that $w_n(t)$ is small in the sense of $E_n(w_n(t))$ and thus small in $H^s$ for small time $t_n$, and the blow-up of $v_n(t_n)$ in $H^s$ forces $u_n(t_n,\cdot)$ to blow up in $H^s$. Here we are not going to write down all the details, but turn to the more general equations: the wave equation of power $3\leq p<5$.

\section{Nonlinear wave equation of power $3\leq p<5$}\label{sec:pnlw}
    The goal of this section is to show that the equation
    \begin{eqnarray*}
        (\ptd^2_t-\Delta)u+|u|^{p-1}u=0
    \end{eqnarray*}
is generically ill-posed in $\cH^s$ for $s\in (0,\fr{3}{2}-\fr{2}{p-1})$. The following proposition is just a quantitative statement of Proposition \ref{thm:main}. 
    \begin{prop}\label{prop:ilpos:gene}
        Given a $3D$ manifold $M$. For any $s\in(0,\fr{3}{2}-\fr{2}{p-1})$ fixed and any given data $(u_0,u_1)\in{C^{\infty}_c(M)}\times{C^{\infty}_c(M)}$ and its corresponding solution $u(t)$ to the equation \eqref{eqn:p}, there exist $\delta>0$, a sequence $(t_n)^{\infty}_{n=1}$ of positive numbers decreasing to zero and a sequence $\big(u_n(t)\big)^{\infty}_{n=1}$ of $C^{\infty}(M)$ functions such that
        \begin{eqnarray}\label{eqn:p:series}
            \begin{cases}
                &(\ptd^2_t-\Delta)u_n+|u_n|^{p-1}u_n=0\\
                &\big(u_n(0),\ptd_tu_n(0)\big)=(u_{0n},u_{1n})
            \end{cases}
        \end{eqnarray}
        with
        \[
            \|(u_{0n},u_{1n})-(u_0,u_1)\|_{\cH^s(M)}\leq{C}\log(n)^{-\delta}\rightarrow0\textrm{ as } n\rightarrow+\infty,
        \]
        but
        \[
            \|u(t_n)\|_{H^s(M)}\geq{C}\log(n)^{\delta}\rightarrow+\infty\textrm{ as }n\rightarrow+\infty.
        \]
    \end{prop}
    \begin{proof}
        We present the proof analogous to that in \cite{BT07a}. We work in a local coordinate near a fixed point of $M$, and will not distinguish this with the Euclidean space. By choosing the initial data $\big(u_{0n},u_{1n}\big)=\big(u_0+\psi_n,u_1\big)$ ( $\psi_n$ to be chosen later), we compare the solution $u_n(t)$ of the equation \eqref{eqn:p:series} to the ODE
        \begin{eqnarray}\label{ode:p:series}
            \begin{cases}
                &\fr{d^2}{dt^2}v_n+|v_n|^{p-1}v_n=0\\
                &\big(v_n(0),\fr{dv_n}{dt}(0)\big)=(\psi_n,0).
            \end{cases}
        \end{eqnarray}
        \begin{lem}
            The solution $V$ to the ODE
            \begin{equation}\label{ode}
                V''(t)+|V(t)|^{p-1}V(t)=0,\ V(0)=1,\ V'(0)=0
            \end{equation}
            is periodic. And by choosing $\psi_n(x)=\kappa_nn^{q_1}\phi(nx)$ ($\phi$ and $\kappa_n$ to be chosen later) with $q_1=\fr{3}{2}-s$, the solution $v_n(t)$ to the ODE \eqref{ode:p:series} has an explicit expression
            \[
                v_n(x,t)=\kappa_nn^{q_1}\phi(nx)V\Big(t\big(k_nn^{q_1}\phi(nx)\big)^{\fr{p-1}{2}}\Big).
            \]
        \end{lem}
        \begin{proof}
            Multiplying the ODE \eqref{ode} by $V$, we see
            \[
                \fr{d}{dt}\Big(\fr{1}{2}|V'(t)|^2+\fr{1}{p+1}|V|^{p+1}\Big)=0,
            \]
            from which we have
            \[
                \fr{1}{2}|V'(t)|^2+\fr{1}{p+1}|V|^{p+1}=\fr{1}{2}.
            \]
            Then by a qualitative analysis, we have $V$ is periodic. The left of the lemma is just a computation.
        \end{proof}
        In order to exploit a deeper property of $v_n$, we need to select the bump function $\phi$ carefully.
        \begin{lem}\label{lem:bump}
            There exists a nontrivial bump function $\phi$ supported in the unit ball $B(0,1)\in\bR^3$ such that
            \[
                \int_{\bR^3}\big|\phi_i\phi_j\phi_k\big|^2\phi^{p-5}dx < \infty,
            \]
            for any possible combination $i,j,k\in\{1,2,3\}$.
        \end{lem}
        By choosing such a bump function $\phi$ obtained in the lemma \ref{lem:bump} and set $t_n=[\log(n)]^{\delta_2}(\kappa_nn^{q_1})^{-\fr{p-1}{2}}$ and $\kappa_n=\log(n)^{-\delta_1}$( $\delta_1$ to be determined later), we have for $t\in[0,t_n]$,
        \begin{eqnarray*}
            \|\Delta(v_n)(t,\cdot)\|_{L^2(M)}&\leq&Ck_nn^{q_1}\big(t_n(k_nn^{q_1})^{\fr{p-1}{2}}\big)^2n^{\fr{1}{2}}\\
            \|\Delta(v_n)(t,\cdot)\|_{H^1(M)}&\leq&Ck_nn^{q_1}\big(t_n(k_nn^{q_1})^{\fr{p-1}{2}}\big)^3n^{\fr{3}{2}}\\
            \|\nabla^kv_n(t,\cdot)\|_{L^{\infty}(M)}&\leq&C\big(t_n(k_nn^{q_1})^{\fr{p-1}{2}}\big)^kk_nn^{q_1+k}\quad\textrm{for}\quad{k=0,1,2}.
        \end{eqnarray*}
        By working on the semi-classical energy $E_n(u)$ defined by
        \begin{equation*}
            E_n(u)\equiv{n^{-q_2}}\sqrt{\|\ptd_tu\|^2_{L^2(M)}+\|\nabla{u}\|^2_{L^2(M)}}
                              +n^{-q_2-1}\sqrt{\|\ptd_tu\|^2_{H^1(M)}+\|\nabla{u}\|^2_{H^1(M)}},
        \end{equation*}
        with $q_2=\fr{5-p}{2}(\fr{3}{2}-\fr{1}{p-1}-s)$, we can show that for every small times, $u_n$ and $v_n$ are close to each other with respect to $E_n$ but these small times are not long enough to drive these both away from each other in $H^s$. Here is the statement
        \begin{lem}\label{lem:diff:clos}
            There exist $\epsilon>0,\ \delta_2>0$ and $C>0$ such that for $n\gg1$ and every $t\in[0,t_n]$, we have
            \[
                E_n(u_n(t)-v_n(t))\leq Cn^{-\epsilon},
            \]
            and hence
            \[
                \|u_n(t)-v_n(t)\|_{H^s(M)}\leq Cn^{-\epsilon}.
            \]
        \end{lem}
        \begin{proof}
            Set $w_n=u_n-v_n$, then $w_n$ solves the equation
            \begin{eqnarray*}
                (\ptd^2_t-\Delta)w_n&=&\Delta{v_n}+|v_n|^{p-1}v_n-|v_n+w_n|^{p-1}(v_n+w_n)\equiv{F}\\
                (w_n(0,x),\ptd_tw_n(0,x))&=&(u_0,u_1).
            \end{eqnarray*}
            By the energy inequality for the wave equation, we get
            \[
                \fr{d}{dt}\big(E_n(w_n(t))\big)\leq{C}n^{-q_2}\|F(t,\cdot)\|_{L^2(M)}+Cn^{-(q_2+1)}\|F(t,\cdot)\|_{H^1(M)},
            \]
            By using the bounds for $\|\Delta(v_n)(t,\cdot)\|_{L^2(M)}$ and $\|\Delta(v_n)(t,\cdot)\|_{H^1(M)}$, we have
            \begin{equation}\label{ine:eng:1}
                \fr{d}{dt}\big(E_n(w_n(t))\big)\leq C[\kappa_nn^{q_1+1/2-q_2}\big(t_n(\kappa_nn^{q_1})^{\fr{p-1}{2}}\big)^3+
                       n^{-q_2}\|G(t,\cdot)\|_{L^2(M)}+n^{-(q_2+1)}\|G(t,\cdot)\|_{H^1(M)}\big)]
            \end{equation}
            where
            \[
                G\equiv|v_n|^{p-1}v_n-|v_n+w_n|^{p-1}(v_n+w_n).
            \]
            Writing for $t\in[0,t_n]$,
            \[
                w_n(t,x)=\int^t_0\ptd_sw_n(s,x)ds,
            \]
            we obtain
            \begin{equation}\label{ene:hk}
                \|w_n(t,\cdot)\|_{H^k(M)}\leq t_n\sup_{0\leq\tau\leq{t}}\|\ptd_tw_n(\tau,\cdot)\|_{H^k(M)}.
            \end{equation}
            In particular
            \[
                \|w_n(t,\cdot)\|_{L^2(M)}\leq{t_n}n^{q_2}e_n\big(w_n(t)\big),
            \]
            where $e_n(w_n(t))\equiv\sup_{0\leq\tau\leq{t}}E_n(w_n(\tau))$. And hence we have
            \begin{eqnarray*}
                \|w_n(t,\cdot)\|_{H^1(M)}&\leq&{C}n^{q_2}e_n(w_n(t)),\\
                \|w_n(t,\cdot)\|_{H^2(M)}&\leq&{C}n^{q_2+1}e_n(w_n(t)).
            \end{eqnarray*}
            Thanks to the Gagliardo-Nirenberg interpolation, we have
            \begin{equation}\label{ene:infty}
                \|w_n(t,\cdot)\|_{L^{\infty}}\leq{C}\|w_n(t,\cdot)\|^{1/2}_{H^2(M)}\|w_n(t,\cdot)\|^{1/2}_{H^1(M)}\leq{C}n^{q_2+1/2}e_n(w_n(t)),
            \end{equation}
             Using \eqref{ene:hk}, \eqref{ene:infty}, the $L^{\infty}$-bound for $v_n$ together with H\"older inequality and the inequality $\big||a+b|^{p-1}(a+b)-|a|^{p-1}a\big|\leq C(|a|^{p-1}+|b|^{p-1})|a|$ for some positive constant $C$, we obtain
            \begin{eqnarray*}
                n^{-q_2}\|G(t,\cdot)\|_{L^2(M)}&\leq&Ct_n\Big((\kappa_nn^{q_1})^{p-1}e_n(w_n(t))+(n^{q_2+1/2})^{p-1}e_n(w_n(t))^p\Big),\\
                n^{-(q_2+1)}\|G(t,\cdot)\|_{H^1(M)}&\leq&Ct_n(\kappa_nn^{q_1})^{\fr{p-1}{2}}t_n\Big((\kappa_nn^{q_1})^{p-1}e_n(w_n(t))+\kappa_nn^{q_1}(n^{q_2+1/2})^{p-2}e_n(w_n(t))^p\Big)\\
                                                   & & +Ct_n\Big((\kappa_nn^{q_1})^{p-1}e_n(w_n(t))+(n^{q_2+1/2})^{p-1}e_n(w_n(t))^p\Big)
            \end{eqnarray*}
            Therefore, coming back to \eqref{ine:eng:1}, we get
            \begin{eqnarray*}
                \fr{d}{dt}\big(E_n(w_n(t))\big)&\leq&C[\kappa_nn^{q_1+1/2-q_2}\big(t_n(\kappa_nn^{q_1})^{\fr{p-1}{2}}\big)^3\\
                                                &&+t_n\Big((\kappa_nn^{q_1})^{p-1}e_n(w_n(t))+(n^{q_2+1/2})^{p-1}e_n(w_n(t))^p\Big)\\ &&+t_n(\kappa_nn^{q_1})^{\fr{p-1}{2}}t_n\Big((\kappa_nn^{q_1})^{p-1}e_n(w_n(t))+\kappa_nn^{q_1}(n^{q_2+1/2})^{p-2}e_n(w_n(t))^p\Big)].
            \end{eqnarray*}
            We first suppose that $e_n(w_n(t))\leq 1$ which hold for small values of $t$ due to the smallness of $e_n(w_n(0))=C{(u_0,u_1)} n^{-q_2}\ll n^{-\epsilon}$ for any $\epsilon$ sufficiently close to zero and the continuity of $e_n(w_n(t))$. We then get
            \[
                \fr{d}{dt}E_n(w_n(t))\leq{C}\kappa_nn^{q_1+1/2-q_2}\big(t_n(\kappa_nn^{q_1})^{\fr{p-1}{2}}\big)^3+C\big(t_n(\kappa_nn^{q_1})^{\fr{p-1}{2}}\big)^2f(n)e_n(w_n(t))
            \]
            where $f(n)\equiv\big(\kappa_nn^{q_1}\big)^{\fr{p-1}{2}}+\log(n)^{-\delta_2}\big(\kappa_nn^{q_1}\big)^{-\fr{p-1}{2}}(n^{q_2+1/2})^{p-1}+\big(\kappa_nn^{q_1}\big)^{1-\fr{p-1}{2}}(n^{q_2+1/2})^{p-2}$. Thanks to the special choices of $q_1$, $q_2$ and the super-criticality of $s$, we have the bound $f(n)\leq (\kappa_nn^{q_1})^{\fr{p-1}{2}}$. Now going through a Gronwall argument for $t\in[0,t_n]$, we obtain
            \begin{eqnarray*}
                e_n(w_n(t))&\leq&\fr{\kappa_nn^{q_1+1/2-q_2}(\log{n})^{3\delta_2}}{(\log{n})^{2\delta_2}(\kappa_nn^{q_1})^{\fr{p-1}{2}}}\times{e^{(\log{n})^{2\delta_2}(\kappa_nn^{q_1})^{\fr{p-1}{2}}t}}\\
                           &\leq&\fr{\kappa_nn^{q_1+1/2-q_2}}{(\kappa_nn^{q_1})^{\fr{p-1}{2}}}\log(n)^{\delta_2}e^{\log(n)^{3\delta}}\\
                           &\leq&\kappa^{1-\fr{p-1}{2}}_nn^{q_1+1/2-q_2-q_1\fr{p-1}{2}}\log(n)^{\delta_2}e^{\log(n)^{3\delta_2}}\\
                           &\leq&\log(n)^{\delta_1\fr{p-3}{2}+\delta_2}e^{\log(n)^{3\delta_2}}n^{s-(\fr{3}{2}-\fr{2}{p-1})}.
            \end{eqnarray*}
            Thus for $\delta_2$ sufficiently small, there exists some $\epsilon>0$ such that
            \[
                E_n(w_n(t))\leq Cn^{-\epsilon}.
            \]
            In particular, one has for $t\in [0,t_n]$
            \begin{equation}\label{ene:h1}
                \|\ptd_tw_n(t,\cdot)\|_{L^2(M)}+\|w_n(t,\cdot)\|_{H^1(M)}\leq{n^{q_2-\epsilon}}=Cn^{\fr{5-p}{2}(q_1-\fr{1}{p-1})-\epsilon}.
            \end{equation}
            Next we have for $t\in [0,t_n]$,
            \[
                \|w_n(t,\cdot)\|_{L^2(M)}\leq{C}t_n\sup_{0\leq\tau\leq{t}}\|\ptd_tw_n(\tau,\cdot)\|_{L^2(M)}\leq{C}\log(n)^{\delta_2+\delta_1\fr{p-1}{2}}n^{q_1(3-p)-\fr{5-p}{2(p-1)}-\epsilon}.
            \]
        Interpolating between this last inequality with the inequality \eqref{ene:h1} yields
        \begin{equation}\label{ene:hs}
            \|w_n(t,\cdot)\|_{H^s(M)}\leq{C}\log(n)^{(1-s)(\delta_2+\delta_1\fr{p-1}{2})}n^{g(s,p)}n^{-\epsilon},
        \end{equation}
        where $g(s,p)\equiv-\fr{p-1}{2}s^2+\fr{7p-15}{4}s-\fr{(p-2)(3p-7)}{2(p-1)}=-\fr{1}{2}\big((p-1)s-2(p-2)\big)\big(s-(\fr{3}{2}-\fr{2}{p-1})\big)$ is non-positive for $s\in(0,\fr{3}{2}-\fr{2}{p-1})$ and $3\leq p<5$. Thus, we finally get
        \[
            \|u_n(t,\cdot)-v_n(t,\cdot)\|_{H^s(M)}\leq Cn^{-\epsilon}.
        \]
        This finishes the proof of lemma \ref{lem:diff:clos}.
        \end{proof}
            Now using lemma \ref{lem:diff:clos}, we have
            \[
                \|u_n(t_n,\cdot)\|_{H^s(M)}\geq\|v_n(t_n,\cdot)\|_{H^s(M)}-Cn^{-\epsilon}.
            \]
            On the other hand, by the representation of $v_n$, we have for $n$ large enough
            \[
                \|v_n(t_n,\cdot)\|_{H^s(M)}\geq{C}\kappa_nn^{q_1}n^{-(3/2-s)}\big(t_n(\kappa_nn^{q_1})^{\fr{p-1}{2}}\big)^s=C[\log(n)]^{-\delta_1+s\delta_2}.
            \]
            Rigorously, this is just a consequence of the following lemma.
        \begin{lem}\cite{BT07a}\label{lem:bt}
            Consider a smooth non constant $2\pi$ periodic function $V$ and two functions $\phi,\psi\in C^{\infty}_c(\bR^d)$ such that $\phi\psi$ is not identically vanishing. Then there exists $C>0$ such that for any $\lambda>1$ and any $s\geq0$
            \[
                \|\psi(\cdot)V(\lambda\phi(\cdot))\|_{H^s(\bR^d)}\geq\fr{\lambda^s}{C}-C.
            \]
        \end{lem}
    \end{proof}

\section{Instantaneous Ill-posedness for the wave equation}\label{sec:inst}
    Now we are ready to prove that the equation \eqref{eqn:p} is actually ill-posed instantaneously for generic initial data
\begin{prop}\label{prop:ins:il}
    Let us fix $s\in]0,\fr{3}{2}-\fr{2}{p-1}[$. Then for any $(u_0,u_1)\in C^{\infty}_c\times C^{\infty}_c$, the equation \eqref{eqn:p}, satisfying in addition the finite speed of propagation, has no solution in $L^{\infty}([-T,T];\cH^s),T>0$ with initial data $(u_{0},u_{1})$.
\end{prop}
\begin{proof}
   Let $\big(\psi_n(x)\big)$ be a sequence of nontrivial bump function such that $\psi_n$ is supported around the point $x_n=(x_{n,1},x'_n=0)$ with $|x_{n,1}|\rightarrow0$ to be specified. As a consequence, $\psi_n$ is supported in the set
   \[
        \{x\in\bR^3\colon|x_1-x_{n,1}|+|x'|\leq\fr{C}{n}\}.
   \]
   Now let $v_n$ be the solution of the ODE \eqref{ode:p:series} with the initial data $\psi_n$. Notice that for any fixed $n$, we have that $\|\psi_n\|_{H^s}\sim \big(\log(n)\big)^{-\delta_1}$. This allow us to consider a sub-sequence $\{n_k\}$ such that $n_k\leq2^{-k}$ and $\|\psi_n\|_{\cH^s}\leq 2^{-k}$. Select $x_{n_k,1}=\fr{1}{k^2}$, then the sets
   \[
        K_{n_k}=\{x\in\bR^3\colon |x_1-x_{n_k,1}|+|x'|\leq \fr{C}{2n}-Ct_{n_k}\}
   \]
   are disjoint with each other for two different $k$'s. And then by Lemma \ref{lem:bt} and Proposition \ref{prop:ilpos:gene}, the solution $u_{n_k}$ to the equation \eqref{eqn:p:series} with initial data $(u_0+\phi_{n_k},u_1)$ satisfy
   \[
        \|u_{n_k}(t_{n_k},\cdot)\|_{\cH^s(K_{n_k})}\geq \log(n_k)^{\alpha}
   \]
   for some positive $\alpha$, as is shown in Proposition \ref{prop:ilpos:gene}. Now consider the initial value problem
   \begin{eqnarray*}
        \begin{cases}
            &(\ptd^2_t-\Delta)u_{\infty}+|u_{\infty}|^{p-1}u_{\infty}=0\\
            &(u_{\infty,0}(0),u_{\infty,1}(0))=(u_0+\sum_{k\geq k_0}\psi_{n_k},u_1)
        \end{cases}
   \end{eqnarray*}
   where $k_0$ is sufficiently large. And thus by the disjointness of $K_n$, the finite speed of propagation and the Sobolev-{Slobodeckij} characterization of fractal Sobolev space on bounded domain (see \cite{marschall1987trace}), we have for arbitrary $t_{n_k}>0$ with $n_k$ sufficiently large
   \[
        \|u_{\infty}(t_{n_k},\cdot)\|_{H^s}\geq \|u_{\infty}(t_{n_k},\cdot)\|_{H^s(K_n)}\sim \|u_{n_k}(t_{n_k},\cdot)\|_{H^s(K_n)}\geq \log(n_k)^{\alpha},
   \]
   and hence
   \[
        {\lim\text{sup}}_{k\rightarrow\infty}\|u_{\infty}(t_{n_k},\cdot)\|_{H^s}=+\infty.
   \]
   To finish the proof of Proposition \ref{prop:ins:il}, it remains to show that $(u_0+\sum_{k\geq k_0}\psi_{n_k},u_1)\in \cH^s-C^{\infty}_c$, which is just a consequence of the selection of $\psi_k$'s and the fact that $\|u_0+\sum_{k\geq k_0}\psi_k\|_{C_0}=\infty$.
   \end{proof}

\bibliographystyle{plain}
\bibliography{prob2topo}

\end{document}